\documentclass[a4paper,10pt, reqno]{amsart}
\usepackage[utf8]{inputenc}

\usepackage{amssymb}
\usepackage{color}
\usepackage[hyphenbreaks]{breakurl}
\usepackage[hyphens]{url}
\definecolor{webgreen}{rgb}{0.1,.1,.8}

\usepackage[colorlinks=true,linkcolor=webgreen,filecolor=webbrown,citecolor=webgreen]{hyperref}%

\usepackage{titleref}
\usepackage{nameref}
 \usepackage{enumitem}
 \newtheorem{question}{Question}

\newcommand{\N}{\mathbb{N}}

\newcommand{\tvname}{A\text{-Triveni triplet}}

 \newtheorem{thm}{Theorem}[section]
 \newtheorem{cor}[thm]{Corollary}
 \newtheorem{lem}[thm]{Lemma}
 \newtheorem{prop}[thm]{Proposition}
 \newtheorem{defn}[thm]{Definition}
 \theoremstyle{remark}

 \numberwithin{equation}{section}
\title{Geometric progressions in  syndetic sets}
\author{Bhuwanesh Rao Patil}
\address{Harish-Chandra Research Institute, HBNI\\ Chhatnag Road, Jhunsi, Allahabad - 211019, Uttar Pradesh, India.}
 \numberwithin{equation}{section}

\begin{document}

% \title{Geometric progressions in  syndetic sets}
% 
% \author{Bhuwanesh Rao Patil}
% \address{Harish-Chandra Research Institute, HBNI,\\ Chatnag Road, Jhunsi, Allahabad - 211019, Uttar Pradesh, India.}
% \email{bhuwaneshrao@hri.res.in, bhuwanesh1989@gmail.com}
% 
% \subjclass{Primary 11B05, 11B75; Secondary 11A05, 11A07, 11D45}
% \keywords{Syndetic set, Geometric progression, Triveni triplet, congruence}
% 
% \date{January 1, 2004}

\begin{abstract}
In order to investigate multiplicative structures in additively large sets,  Beiglb\"{o}ck et al. raised a significant open question as to  whether or
not every subset of the natural numbers with bounded gaps (syndetic set) contains arbitrarily long geometric progressions.
A result of Erd\H{o}s implies that syndetic sets contain a $2$-term geometric progression with integer common ratio,  but we still do
not know if they contain  such a progression with  common ratio being perfect square.  In this article, we  prove that for
each $k\in \mathbb{N}$, 
a syndetic set  contains $2$-term geometric progressions  with common ratios of the form  $n^kr_1$ and $p^kr_2$, where  $p\in\mathbb{P}$ (the set of primes), $n$ is a composite number, $r_1\equiv 1 \pmod{n}$, $r_2\equiv 1\pmod{p}$ and $r_1,r_2\in \mathbb{N}$. We also show that  2-syndetic sets (sets  with bounded gap two) contain  infinitely many $2$-term 
geometric progressions  with their respective common ratios being perfect squares.
\end{abstract}

%%% ----------------------------------------------------------------------
\maketitle
%%% ----------------------------------------------------------------------
%\tableofcontents
\section{Introduction}\label{intro}
 Previous research (e.g. \cite{Beiglbock}, \cite{Bergelson1}, \cite{Bergelson} and \cite{hindman})  establishes that  sets which are
 large in any of several multiplicative senses must have substantial additive structure.
 For example, a multiplicatively piecewise syndetic set\footnote{For definition of multiplicatively piecewise syndetic set, see \cite{Beiglbock}.} in $\mathbb{N}$ must contain arbitrarily long arithmetic
 progressions \cite[Theorem 1.3]{Beiglbock}. However, very little appears to be known as the existence of multplicative structures in 
 ``additively large'' sets.
  If we define ``additive largeness" as having positive upper asymptotic density\footnote{Upper asymptotic density of a set $A\subset \mathbb{N}$ is defined by $\overline d(A):=\limsup\frac{|A\cap[1, n]|}{n}$}, then one can observe that there are
  additively large sets that do not contain three term geometric progressions. For instance, the set of  square-free numbers is additively large, because it has positive upper asymptotic density, but does not 
  contain a configuration of the form $\{x,xr^2\}$. One may still ask if the property of containing multiplicative structures holds for interesting classes of sets that are additively large. This brings us to the following definition: 
 \begin{defn}[Syndetic set]
 If $l\in\mathbb{N}$, then   $A\subset \mathbb{N}$ is called $l$-syndetic set if $A$ has a non-empty intersection with every set of $l$ consecutive natural numbers. A subset of the natural numbers which is  $l$-syndetic for some $l\in \mathbb{N}$, is known as a syndetic set.
 \end{defn}
 An infinite arithmetic progression is the simplest example of a syndetic set. Plainly, syndetic sets have positive upper asymptotic density and thus are additively large.
 Beiglb\"{o}ck et al. \cite{Beiglbock}  recognized the significance of looking for
 geometric progressions  in syndetic sets in order to study multiplicative structures in additively large sets. They asked the following question.
 \begin{question}\label{question1}
  If $A$ is  syndetic, do there exist $x,y\in \mathbb{N}$ such that $\{x,xy, xy^2\}$ $\subset A$?
 \end{question}
 In recent work\cite{Glasscock}, Daniel Glasscock et al. gave some evidence towards an affirmative answer to this question by showing that many syndetic sets
 of dynamical origin contain arbitrarily long geometric progressions. 
The fact that syndetic sets contain  $2$-term geometric progression with integer common ratio  is a consequence of the following propositions,
namely Proposition \ref{erdos} for dense sets and Proposition \ref{aps_beiglbock} for  additively piecewise syndetic sets.
\begin{prop}[Erd\H{o}s \cite{Erdos}]\label{erdos}
 Suppose that $A$ is  a subset of natural numbers such that lower asymptotic density 
 $\underline d(A):=\liminf\frac{|A\cap[1, n]|}{n}>0$. Then $A$ contains a $2$-term geometric progression with integer common ratio.
\end{prop}

 \begin{defn}[Additively piecewise syndetic set]
Let  $A\subset \mathbb{N}$. Then $A$ is called additively piecewise syndetic set if there exists  $l\in \mathbb{N}$ 
such that  for every $n\in \mathbb{N}$,  
there is a sequence $(x_i)_{i=1}^{n}$ in $A$ satisfying $0<x_{i+1}-x_{i}\leq l$  $\forall ~i\in [1,n-1]$. 
For example, every syndetic set is additively  piecewise syndetic set. 
\end{defn}
 \begin{prop}\cite[Corollary 2.17]{Beiglbock}\label{aps_beiglbock}
  If $A$ is an additively  piecewise syndetic set, then there exists a sequence $(y_{n})_{n=1}^{\infty}$ in $\mathbb{N}\backslash \{1\}$ such that for each $n\in \mathbb{N}$,  $ y_{n+1}\equiv 1  \pmod {\prod_{i=1}^{n} y_{i}}$  and $\prod_{i=1}^{n} y_{i} \in$ $A$. 
 \end{prop}
 \noindent
In other words, an additively piecewise syndetic set  contains configurations of the type $\{x,xy\}$ for some $x,y\in \mathbb{N}$. But there exists some additively piecewise syndetic set which does not contain configurations of the type $\{x,xy,xy^2\}$ with $x,y\in \mathbb{N}$.
One can get this type of set inside the collection of  thick sets where a thick set is a subset of the natural numbers containing arbitrarily long intervals in $\mathbb{N}$. 
Using the fact that every thick set is additively piecewise syndetic,  Proposition
\ref{Thick_set_3gp2} 
guarantees the existence of a piecewise syndetic set not containing configurations of the type
$\{x,xy, xy^2\}$ with $x,y\in \mathbb{N}$. 
   \begin{prop} \cite[Theorem 3.5]{Beiglbock}\label{Thick_set_3gp2}
   There is a thick subset $A$ of $\mathbb{N}$ such that  there do not exist $a\in A$ and $r\in \mathbb{Q}\setminus\{1\}$ such that $ar\in A$ and $ar^2\in A$.
   \end{prop}
 
The following weaker version of Question \ref{question1} is still an open question. 
\begin{question}\label{question2}
  If $A\subset \mathbb{N}$ is  syndetic, do there exist $x,y\in \mathbb{N}$ such that $\{x, xy^2\}\subset A$?
\end{question}

Our first result relates to the above question and gives more information about $2$-term geometric
progressions with integer common ratios in syndetic sets.

\begin{thm}\label{thm1}
 Let $k\in \mathbb{N}$ and $H_0\in\{\mathbb{P}, \mathbb{N}\setminus\mathbb{P}\}$. Then any syndetic set contains a $2$-term geometric 
 progressions  with common ratio $n^kr$ for some $n\in H_0\setminus\{1\}$ and  $r\in \mathbb{N}$ with $r\equiv 1 \pmod{n}$.
 \end{thm}

  The next result  of this paper confirms an affirmative answer to
  Question \ref{question2} in the case of $2$-syndetic sets.

\begin{thm}\label{thm2}
 A $2$-syndetic set contains infinitely many $2$-term geometric progressions whose common ratios are  perfect squares. 
\end{thm}
The proof of Theorem \ref{thm1} uses Chinese remainder theorem extensively.
In Section \ref{section2},
 we show that generating pairwise prime sets [see Definition \ref{pairwise_prime}] in a syndetic set is  enough for finding configurations as required 
 in Theorem \ref{thm1}. Section \ref{section3} describes about Triveni triplets
 [see Definition \ref{triveni}] to understand  pairwise prime subsets of a syndetic set.
 Zorn's lemma guarantees the  existence of Triveni triplets of order one with respect to a given syndetic set and then repeated use 
 of Chinese remainder 
 theorem  at various stages generates Triveni triplets of higher order. Using these observations, Section \ref{section4} explains  the proof of Theorem \ref{thm1}.
  Section \ref{section5} describes the proof of Theorem \ref{thm2}
 by producing infinitely many explicit geometric progressions.
\subsection*{Notation}\label{notation}

Let $\mathbb{Q}$, $\mathbb{N}, \mathbb{P}$ and $\mathbb{Z}$ denote, respectively, the set of rational numbers, the 
set of positive integers, the set of prime numbers and the set of integers.
For $k\in \mathbb{N}$ and $H\subset \mathbb{N}$, $\mathcal{S}_{k,l,H}$ denotes the collection of $l$-syndetic sets that
contain a
configuration of the form $\{x,xn^kr\}$, where 
$n\in H$, $r\in \mathbb{N}$ satisfying $r\equiv 1\pmod{n}$.
Let $[a, b] :=$ 
$\{x \in \mathbb{Z}: a \leq x \leq b\}$. Let $\subsetneq$  denote the ``proper subset of''.
For $A\subset \mathbb{N}$ and $x\in\mathbb{N}$, $xA=Ax:=\{nx\colon n\in A\}$. 
 
\section{Chinese remainder theorem and syndetic sets } \label{section2}
\begin{defn}[Pairwise prime set]\label{pairwise_prime}
 Let $B\subset \mathbb{N}$. Then $B$ is called a pairwise prime set if $\gcd(a,b)=1 ~\forall~ a,b\in B$ with $a\neq b$. 
\end{defn}

 Using the Chinese Remainder Theorem\cite{Nathanson}, the next lemma helps us to prove   
 Theorem  \ref{thm1} in the case of syndetic sets containing  pairwise prime sets with arbitrarily large cardinalities. 
\begin{lem}\label{Main_prop1}
 Let $h\in \mathbb{N}$, let $m_{1}, m_{2}, ..., m_{h}$ be pairwise co-prime integers in $\mathbb{N}$ and let $t_{1}, t_{2}, ..., t_{h}$
 be arbitrary elements in $\mathbb{N}\cup \{0\}$.
 Then $\exists$ $u_0\in \mathbb{N}$ such that 
  if $u_t=u_0+t\left( \prod_{i=1}^hm_i^2\right) $ with $t\in \mathbb{N}$, then there exists $(r_{i,t})_{i=1}^{h}$ in  $\mathbb{N}$
  satisfying  $u_t+t_{i}=r_{i,t}m_{i}$ and $r_{i,t} \equiv 1\pmod{m_{i}}$ $~\forall~ i\in [1, h]$.
\end{lem}
\begin{proof}
 Consider the congruences
 $ x+ t_{i}\equiv m_i \pmod {m_{i}^2} ~\forall~ i\in [1, h].$

 Since  $\{m_i\colon i\in [1,h]\}$ is a pairwise prime subset in $\mathbb{N}$ and
 $\{ t_{i}\colon i\in[1,h]\}\subset \mathbb{N}\cup\{0\}$, the Chinese remainder theorem ensures the existence of  $u_0\in \mathbb{N}$ 
 such that if  $u_t=u_0+tw^2$  for  $w:=\prod_{i=1}^{h} m_{i}$ and $t\in \mathbb{N}$, then
  $u_t+t_{i}\equiv m_{i} \pmod{m_i^2} ~\forall~ i\in [1, h]. $ 
 Therefore there exists a sequence $ (r_{i,t})^{h}_{i=1}$ in $\mathbb{N}$ such that 
 $u_t+t_{i}=r_{i,t} m_{i} \text{ and } r_{i,t}\equiv 1\pmod{m_i} ~\forall~ i\in [1, h].$
\end{proof}

From the above lemma, we get the following important corollary which will be used recursively in the proof of Theorem \ref{thm1}.
\begin{cor}\label{Main_prop1_Maincor1}
 Let $h,k\in \mathbb{N}$, $A\subset\mathbb{N}$ be an infinite set and $B\subset \mathbb{N}$ be a pairwise prime set such that $|B|=h$ and $uB\subset A$  for some $u\in\mathbb{N}$.
 If $H$ be a pairwise prime subset of $\mathbb{N}$ such that $|H|=|B|$ and $\gcd(a,b)=1~\forall~ a\in H, b\in B$, then
 at least one of the following is true.
 \begin{enumerate}
  \item $A$ contains a configuration of the type $\{x,xn^kr\}$ for some $x,r\in \N$, $n\in H$ and $r\equiv 1\pmod{n}$.
  \item  There exists $z'\in \mathbb{N}$  such that $u[z_t, z_t+h-1]\cap A=\varnothing$ $ ~\forall ~t\in \N$ where $z_t=z'+t\left(\prod_{x\in B}x^2\right)\left(\prod_{x\in H}x^{2k}\right).$
 \end{enumerate}
\end{cor}
\begin{proof}
Let $B=\{x_1,x_2,\cdots ,x_h\}$. Since $|B|=|H|$, there exists a bijective map $f\colon B\rightarrow H$. Define
$m_i:=x_if(x_i)^k~ \forall ~ i\in [1,h].$

 Since  $H$ and $B$ are  pairwise prime sets satisfying $\gcd(a,b)=1~\forall~ a\in H, b\in B$, we get that the set $\{m_i\colon i\in [1,h]\}$ is a pairwise prime set.
So,  Lemma \ref{Main_prop1} gives $z'\in \mathbb{N}$ which satisfies the property that if
 $z_t=z' + t\left(\prod_{i=1}^{h}m_i^2\right)$ for $t\in \mathbb{N}$, 
then there exists a sequence $(r_{i,t})_{i=1}^{h}$ in $\mathbb{N}$ such that 
\begin{equation}
z_t+i-1=r_{i,t}m_i~~\text{ and }~ ~ r_{i,t}\equiv 1\pmod{m_i} ~\forall ~ i\in [1,h].  \label{eq4}
\end{equation}
Since $uB\subset A$ for some $u\in\mathbb{N}$, we have $y_i:=ux_i\in A$  $ ~\forall ~i\in [1,h]$. Hence, applying the definitions
of $y_i$ and $m_i$
 in the   equation \eqref{eq4}, we have
$$uz_t+u(i-1)=y_i f(x_i)^k r_{i,t} ~~\text{ and }~ ~ r_{i,t}\equiv 1\pmod{f(x_i)} ~\forall ~ i\in [1,h].$$

If $u[z_t, z_t+h-1]\cap A \neq \varnothing$ for some $t\in \mathbb{N}$, then $\exists ~$ $j\in [1,h]$ such that 
$y_j f(x_j)^k r_{j,t}\in A$ and $r_{j,t}\equiv 1\pmod{f(x_j)}$ where $y_j\in A$ and $f(x_j)\in H$. This completes the proof.
 \end{proof} 
For any finite pairwise prime set $B$  and  an infinite pairwise prime subset of $\mathbb{N}$ (say $H_0$), there exists pairwise prime set $H\subset H_0$
 such that $|H|=|B|$ and $\gcd(a,b)=1 \ \forall \ a\in H, b\in B$. Then applying $u=1$ in Corollary  \ref{Main_prop1_Maincor1},
we immediately get the following proposition.
\begin{prop}\label{main_cor1}
 Let $k,l\in\mathbb{N}$, $H_0$ be an infinite pairwise prime subset of $\mathbb{N}$ and $A$ be an $l$-syndetic set. If  $B$ is a pairwise prime 
 subset of $A$ such that   $|B|\geq l$,
 then $A\in \mathcal{S}_{k,l,H_0}$. 
\end{prop}
 If $H_0 = \mathbb{P}$ or an infinite pairwise prime subset of the set of composite numbers, we get Theorem \ref{thm1}
 for those  $l$-syndetic sets which contain a pairwise prime set $B$ with $|B|\geq l$.

\section{Triveni triplets and syndetic sets}\label{section3}
 In the previous section, we observed that our proof of Theorem \ref{thm1} depends on the study of  pairwise prime subsets of
 syndetic sets. For 
  better understanding of these pairwise prime sets, we define Triveni triplet as follows. For $l\in \mathbb{N}$ and $p\in\mathbb{P}$, denote
  $$r(p,l):=\max\{t\in\mathbb{N}\cup\{0\}\colon p^t\leq 2l+1 \},$$
\begin{center}
 
$T(l):=\left\{\prod_{p\in \mathbb{P}\cap [2,2l+1]} p^{r_{p}}\colon 0\leq r_p\leq r(p,l) \right\}$.
\end{center}
\begin{defn}[Triveni triplet]\label{triveni}
 Let $l,h\in\mathbb{N}$, $A\subset \N$ and  $F\subset T(l)\setminus\{1\}$.
 Then $(F,h,l)$ is called a Triveni triplet with respect to the set $A$ if there exists a sequence of pairwise prime sets  $(B_{u})_{u\in F}$ such that
 \begin{enumerate}
  \item $|B_{u}|=h$ and $uB_{u}\subset A$ for each $u\in F$ 
  \item  For distinct $u,v\in F$, $\gcd(x,y)=1 ~\forall ~x\in B_u$ and ~$ y\in B_v$.
 \end{enumerate}
 
 Triveni triplets with respect to the set $A$  are called $\tvname$s. $|F|$ is called the order of the $\tvname$ $(F,h,l)$.
\end{defn}

\subsection{Triveni triplets of order one with respect to syndetic sets}
  One can produce Triveni triplets of order one with respect to  most of the syndetic sets using Zorn's lemma\cite{Halmos}.
\begin{lem}\label{Main_prop2lemma1}
 Let $r>0$ and $A \subset \mathbb{N}$. Let $\mathbb{M}_r$ be the collection of pairwise prime subsets $B$ of $\mathbb{N}$ satisfying 
 $rB$ = $\{ rx : x\in B \} \subset A$. Then $\exists$ $B_{r}\in$ $\mathbb{M}_r$ such that if $C\supset B_{r}$ and $C\in$ $\mathbb{M}_r$, 
 then $C$ = $B_{r}$. Here $B_{r}$ is  a maximal element of $\mathbb{M}_r$.
 \end{lem}
\begin{proof}
 Let $\alpha$ be a chain in  the partially ordered set $(\mathbb{M}_r, \subset)$. Then the union of every elements of $\alpha$ belongs to the set $\mathbb{M}_r$. Hence, Zorn's lemma guarantees the existence of a maximal element. 
\end{proof}
In the above lemma,  $B_r$ may be an empty set. But certainly $B_r\neq\varnothing$ for some $r\in \mathbb{N}$. In particular, $B_1\neq\varnothing$. The next proposition  
deals with the existence of an infinite  pairwise prime set $B_r$ for some $r\in[1,l]$ with respect to those
$(2l+1)$-syndetic sets which do not contain at least two elements of 
$x\mathbb{N}$ for each $x\in \mathbb{N}$.
\begin{prop}\label{Main_prop2}
   Let $A$ be a  $(2l+1)$-syndetic set
   with $|x\mathbb{N}\setminus A|\geq 2 ~\forall ~ x\in \mathbb{N}$.
 Then $\exists$ $d\in[1,l]$ and an infinite pairwise prime set $B\subset \mathbb{N}$ such that  $dB \subset A$.
 \end{prop}
 \begin{proof}
Let $B_r$ be a maximal pairwise prime set satisfying  $rB_r\subset A\setminus [1,l]$ for each $r\in[1,l]$. Lemma \ref{Main_prop2lemma1}
 assures the existence of such $B_r$.
 
 By way of contradiction,  assume that $B_{r}$ is finite for each $r \in [1,l]$. 
 \noindent
 Define
  \begin{equation*}
  C := \displaystyle\bigcup_{r=1}^{l}rB_{r}= \{ u_{1}, u_{2}, ..., u_{m}\}\subset A \text{ and } t:=\prod_{i=1}^{m}u_i          .                                                 
 \end{equation*}
Clearly, $C\neq \varnothing$ and $t>l$ as $B_1\neq \varnothing$. 
Since $|t\mathbb{N}\setminus A|\geq 2$,
 $\exists$ $s\in \mathbb{N}\setminus\{1\}$ such that $st\notin A$. 
 Then the $(2l+1)$-syndeticity of $A$ ensures that  $z_{j}\in A$ for some $j\in [-l,l]\setminus 0$ where
$z_{i}:= st+i$  $~\forall ~i\in [-l,l]$. Note that $z_j>t$ as $j\geq -l$, $s>1$ and $t> l$.  Define 
$$d: = \max\{r\in [1,l]:  B_{r}\neq\varnothing \text{ and } r\mid j\}.$$
Clearly, $d$ is well define as $B_1\neq \varnothing$. Since $dx$ divides $t$ for all $x\in B_d$ and $z_j>t$, we have $z_j>dx$ $\forall ~ x\in B_d$.
 
 Let $x\in B_d$ and $d^{'}= \gcd(dx,z_{j})$. By the definition of $t$, $dx\mid t$. It follows that $d'=\gcd(dx,j)$. 
 Using $x\in B_d$ and  $d^{'}\mid dx$, we get that
the pairwise prime set $Z:=\{\frac{dx}{d^{'}}\}\subset \mathbb{N}$ satisfies $d'Z\subset A\setminus [1,l]$ and so $B_{d^{'}}\neq\varnothing$.
Since $d\mid j$ and $d'\mid j$, the definition of $d$ gives $d\mid d'$ and $d'\leq d$. Hence $d= \gcd(dx,z_{j})$.

 Therefore, we get  $z_j\in A$ such that $\gcd(dx,z_{j})=d$ and $z_j>dx$ for each $x\in B_d$. It gives us 
 pairwise prime set  $Y= B_{d}\cup \{\frac{z_{j}}{d}\}$ satisfying $dY\subset A\setminus[1,l]$ and $B_d\subsetneq Y$. This set $Y$ contradicts the maximality of $B_d$.
So, we get a contradiction to the  assumption that $B_{r}$ is finite for each $r\in [1,l]$. Hence, there exists $ r\in [1,l]$ such that $B_{r}$ is an 
infinite pairwise prime set.
 \end{proof}
Therefore, Triveni triplets of order one can be found as a corollary of the above proposition in the following way.%of the above proposition.

\begin{cor}\label{Existence_triveni_1}
  Let $k,l\in \mathbb{N}$, $H_0$ be an infinite 
 pairwise prime subset of $\mathbb{N}$ and 
$A$ be a  $(2l+1)$-syndetic set such that $A\notin \mathcal{S}_{k,2l+1,H_0}$.
Then there exists $d\in [2,l]$ such that
  $(\{d\},h,l)$ is an $\tvname$ for each $h\in \mathbb{N}$.
\end{cor}
\begin{proof}
 Suppose for each  $d\in [2,l]$, there exists $h_d\in\mathbb{N}$ such that
  $(\{d\},h_d,l)$ is  not an $\tvname$. Applying Proposition \ref{Main_prop2}, we get that either
  $A\supset x\mathbb{N}\setminus\{xy\}$ for some $x,y\in \mathbb{N}$ or there exists an infinite pairwise prime subset $B$
  satisfying $B\subset A$. Therefore, Proposition \ref{main_cor1} concludes the result.  
\end{proof}

\subsection{Triveni triplets of higher order with respect to  syndetic sets}

   Now we shall see a procedure for generating Triveni triplets of higher order with respect to syndetic sets. First we will
   prove some necessary results to demonstrate the procedure.
Due to  the next lemma, one can construct   sets with arbitrary large cardinalities in which the $\gcd$ of any two distinct elements belongs to the set $T(l)$.
Note that $\gcd$ of any two distinct elements in an interval with cardinality $(2l+1)$ belongs to the set $T(l)$. 
\begin{lem}\label{Higher_triveni_mainlemma1}
For $l\in \mathbb{N}$, there exist an increasing function $c\colon \mathbb{N}\longrightarrow \mathbb{N}$ and a strictly increasing  sequence of positive integers $(x_{i})_{i=1}^{\infty}$ such 
 that   $x\leq c(h)$ and $\gcd(x,y)\in T(l)$ $\forall$ $h\in \mathbb{N}$, $x,y\in S_{h}$ with $x\neq y$ where 
 $S_{h}=\cup_{i=1}^{h}[x_i, x_{i}+2l].$
\end{lem}
\begin{proof}
Define $x_1:=1$ and $c(1):=2l+1$. Then $\gcd(a,b)\in T(l)$ for $a,b\in S_1$.

For given $h\in \mathbb{N}$, suppose there exist a sequence $(x_i)^{h}_{i=1}$ in $\mathbb{N}$ and $c(h)\in \mathbb{N}$  such that 
$a\leq c(h)$ and $\gcd(a,b)\in T(l)$ $\forall$ $a,b\in S_h$ with $a\neq b$. To complete the proof by induction,
we need to generate $x_{h+1}$ and $c(h+1)$ such that $\gcd(x,x_{h+1}+j)\in T(l)$, $x_h<x_{h+1}\leq c(h+1)-2l $ and $c(h)\leq c(h+1)$
$\forall$
$x\in S_h $ and $ j\in [0, 2l]$. For this purpose, we define
 $$m:=\prod_{x\in S_h}x \text{ and }
 Y:=\{p^r\colon p\in \mathbb{P},~ p^r\mid m \text{ and } p^{r+1}\nmid m\}$$
 $$x_{h+1}:= 1+\prod_{u\in Y}u \text{ and } c(h+1):=\max\left\{2l+x_{h+1}, ~c(h)\right\}. $$

The definition of  $c(h+1)$ ensures $x_{h+1}\leq c(h+1)-2l $ and $c(h)\leq c(h+1)$. Since $x_h$ divides $m$, the definition of $x_{h+1}$ gives us $x_h < x_{h+1}$.

Let  $x\in S_h $ and $j\in [0,2l]$. To show that  $\gcd(x,x_{h+1}+j)\in T(l)$, let $q$ be a prime divisor of $x$
such that
$q^{r(q,l)+1}\mid x$. Since $x\in S_h$,
it follows that $q^{r(q,l)+1}\mid m$ which is followed by $q^{r(q,l)+1}\mid u$ for some $u\in Y$. 
Then $x_{h+1}+j \equiv j+1\pmod{q^{r(q,l)+1}}$ by  the definition of $x_{h+1}$. Hence, $q^{r(q,l)+1} \nmid x_{h+1}+j$ because of the fact that  $j+1\in [1,2l+1]$ but $q^{r(q,l)+1} > 2l+1$ by the definition
 of $r(q,l)$. Hence $q^{r(q,l)+1}\nmid \gcd(x,x_{h+1}+j)$. Since $r(p,l)=0$ $\forall$ $p\in \mathbb{P}\cup [2l+2,\infty)$, we have
 $\gcd(x,x_{h+1}+j)\in T(l)$ by definition of $T(l)$.
\end{proof}
In the above lemma, the definition of the element $x_{h+1}$ is inspired by the application of the Chinese remainder theorem on the congruences $x\equiv 1\pmod{u} ~\forall~ u\in Y$. One can choose any positive integer which satisfies
these congruences and is  greater than $1$. 
 
\begin{cor} \label{interval_gcd_in_Tl}
There exists  a map $m\colon \mathbb{N}\times \mathbb{N}\rightarrow \mathbb{N}$ such that
for $h,l,n\in\mathbb{N}$, an
 interval $[n, n+m(h,l)]$
contains a set $S$ with the following properties.
\begin{enumerate}[label=(\roman*)]
 \item \label{interval_gcd_in_Tl_condition1} $S=\displaystyle\bigcup_{i=1}^{h}[x_i, x_i+2l]$   for some strictly increasing sequence $(x_i)_{i=1}^{h}$ in $\mathbb{N}$.
 \item \label{interval_gcd_in_Tl_condition2} $\gcd(a,b)\in T(l) ~\forall ~ a,b\in S$ with $a\neq b$.
\end{enumerate}

\end{cor}
\begin{proof}
For given  $l,h\in\mathbb{N}$, Lemma \ref{Higher_triveni_mainlemma1} gives a strictly increasing sequence
$(y_i)_{i=1}^{h}$ in $\mathbb{N}$ and a positive integer $c(h)$ such that 
 $a\leq c(h)$ and $ \gcd(a,b)\in T(l) ~\forall ~a,b\in R \text{ with $a\neq b$ }$
 where 
 $R=\displaystyle\cup_{i=1}^{h} [y_i, y_i+2l]$.
 Hence, define a map $m\colon \mathbb{N}\times \mathbb{N}\rightarrow \mathbb{N}$ such that
 $m(h,l):=L+c(h)$
  where $L=\prod_{q\in \mathbb{P}\cap [2,c(h)]}q^{r(q,l)+1}$.

 Let $n\in \mathbb{N}$ and $a_0\in [n+1, n+L]$ be an integer divisible by $L$. Then  we shall show that the set
 $S:=\{a_0+u\colon u\in R\} $ satisfies $S\subset[n,n+m(h,l)]$ along with  the properties \ref{interval_gcd_in_Tl_condition1} and \ref{interval_gcd_in_Tl_condition2}
 given in the statement of Corollary \ref{interval_gcd_in_Tl_condition1}. By construction, one can note that property \ref{interval_gcd_in_Tl_condition1} is obvious. Using the fact that 
  $a\leq c(h)~\forall ~a\in R$ and $a_0\in [n+1, n+L]$, we get that $S\subset[n,n+m(h,l)]$. 
 
  To prove property \ref{interval_gcd_in_Tl_condition2}, Let $v_1,v_2\in S$ with $v_1\neq v_2$. Then for each $ i\in \{1,2\}$,
 $v_i=a_0+u_i$ for some $u_i\in R$. By way of contradiction, assume that $p$ be a prime  such that $p^{r(p,l)+1}\mid \gcd(v_1,v_2)$.
 Since $r(p,l)\in \mathbb{N}\cup\{0\}$, It follows that  $p\mid u_1-u_2$. This implies $p\leq c(h)$  because $u_1,u_2\in [1,c(h)]$.
 Then, $p^{r(p,l)+1}\mid a_0$ by applying the definitions of $a_0$ and $L$.
It follows that $ p^{r(p,l)+1}\mid\gcd(u_1,u_2)$ by using assumption that
 $ p^{r(p,l)+1}\mid \gcd(v_1,v_2)$. But, this contradicts the fact that  $\gcd(u_1,u_2)\in T(l)$. So we get a contradiction to the assumption that
 $p^{r(p,l)+1}\mid \gcd(v_1,v_2)$.
 
 Therefore,  $p^{r(p,l)+1}\nmid \gcd(v_1,v_2)$ for each $p\in\mathbb{P}$ and hence
 $\gcd(v_1,v_2)\in T(l)$ because $r(p,l)=0$ $\forall$ $p\in [2l+2,\infty)$.
\end{proof}
 The next lemma solves a Diophantine problem using Chinese remainder theorem.
Define $C(l):=(2l+1)^{2l+1}$ and note that  $u< C(l)~\forall~ u\in T(l)$.
\begin{lem}\label{Higher_triveni_mainlemma2}
 Let $(a_i)_{i=1}^{n}$ and $(u_i)_{i=1}^{n} $ be  sequences in $\N\cup\{0\}$ and $T(l)$ respectively.
 If $X=\{x_1,x_2,\cdots x_n\}$
 be a pairwise prime subset of $\mathbb{N}$,
 then there exist $z\in \mathbb{N}$, a sequence $(r_i)^{n}_{i=1}$ in $[0,C(l)]$ and a sequence $(t_i)^{n}_{i=1}$ in $\mathbb{N}$ such that
 \begin{equation} z+r_i=a_i+ t_ix_iu_i ~\forall~ i\in [1,n]. \label{Outcome_Higher_triveni_lemma2}\end{equation}
\end{lem}
\begin{proof}
 Let $u=\text{lcm}(u_1,u_2,\cdots , u_n)$. Then  $u\in T(l)$ and so $u<C(l)$. 
 For each $i\in [1,n]$, choose non-negative integers $r_i$ and $b_i $ such that $a_i=b_iu+r_i$ and $0\leq r_i<u$. Note that
 $r_i\in [0,C(l)]$ as $u< C(l)$.
 Since $X$ is a pairwise prime set, the  Chinese remainder theorem gives 
   $b\in \mathbb{N}$ and a sequence  $(v_i)_{i=1}^{n}$ in $\mathbb{N}$ such that $b= b_i+ v_ix_i$ $\forall$  $i\in [1,n]$.
  Hence, $t_i:=\frac{v_iu}{u_i}$  and $z:=bu$  satisfy the required
  equation in  \eqref{Outcome_Higher_triveni_lemma2}.
\end{proof}
Using the above corollaries and lemmas, the next two propositions demonstrate the  complete procedure to generate Triveni triplets of higher
order with respect to syndetic sets. Define the  map $\Lambda\colon \mathbb{N}\times \mathbb{N}\rightarrow \mathbb{N}$ such that
$$\Lambda(h,l)=m(h,l)+2C(l) ~\forall~ h,l\in \mathbb{N}$$
where the map  $m\colon \mathbb{N}\times \mathbb{N}\rightarrow \mathbb{N}$ is taken from  Corollary \ref{interval_gcd_in_Tl}.
For $l\in\mathbb{N}$ and  $F\subset T(l)$, define
$$D(l):=|T(l)| \text{ and }
Mul(F):=\{v\in T(l)\colon \exists~ u\in F \text{ such that } u|v\}.$$ 
\begin{prop}\label{Higher_triveni_Main_prop}
 Let $k,l,h\in \mathbb{N}$, $A$ be a $(2l+1)$-syndetic set and $H_0$ be an infinite 
 pairwise prime subset of $\mathbb{N}$. 
 If  $(F,\Lambda(D(l)h,l),l)$ is  an $\tvname$ and 
  $A\notin \mathcal{S}_{k,2l+1,H_0}$, then $\exists ~ w\in T(l)\setminus Mul(F)$ and a pairwise prime set $C_w$ such that  $|C_w|=h$ and  $wC_w\subset A$. 
\end{prop}
\begin{proof}
 Since  $(F,\Lambda(D(l)h,l),l)$ is an $\tvname$, there exists a sequence of pairwise prime sets $(B_u)_{u\in F}$ satisfying $|B_u|=\Lambda(D(l)h,l)$ and $uB_u\subset A$ for each $u\in F$ 
 such that 
for distinct $u,v\in F$, 
$$\gcd(x,y)=1 ~\forall ~x\in B_u \text{ and} ~ y\in B_v.$$

Let $B=\displaystyle \cup_{u\in F}B_u$. Since $H_0$ is an infinite pairwise prime set, there exists a sequence of 
pairwise prime subsets of $H_0$ (say $(H_u)_{u\in F}$) satisfying $|B_u|=|H_u|$  $\forall ~u\in F$
such  that
for distinct $u,v\in F$, 
$$\gcd(x,y)=\gcd(x,b)=1 ~\forall ~x\in H_u, ~ y\in H_v \text{ and  } b\in B.$$

Given that $A$ is a $(2l+1)$-syndetic set satisfying $A\notin \mathcal{S}_{k,2l+1,H_0}$.
Applying Corollary
\ref{Main_prop1_Maincor1}, $\exists$ a sequence  $(z_u)_{u\in F}$ in $\mathbb{N}$ such that if 
$z_{u,t}=z_u+t\left(\prod_{x\in B_u}x^2\right)\left(\prod_{x\in H_u}x^{2k}\right)$ with  $t\in \mathbb{N}$, then
 \begin{equation}
 u[z_{u,t},z_{u,t}+ \Lambda(D(l)h,l)-1]\cap A=\varnothing ~\forall~ t\in \mathbb{N}, u\in F.\label{Higher_triveni_Main_propeq1}
 \end{equation}
By Lemma \ref{Higher_triveni_mainlemma2}, there exist  sequences $(r_u)_{u\in F}$ in $[0, C(l)]$, $(t_u)_{u\in F}$ in $\mathbb{N}$ and $z\in \mathbb{N}$ such that
$z+r_u=uz_{u,t_u} ~\forall ~u\in F.$ Then equation \eqref{Higher_triveni_Main_propeq1} guarantees that
$ u\mathbb{N}\cap A\cap  [z+r_u, z+r_u+u\Lambda(D(l)h,l)-u]=\varnothing$  $~\forall ~u\in F.$ 
Since $r_u\in[0,C(l)]$ and $1<u<C(l)$ $\forall$ $u\in F$, it follows that $u\mathbb{N}\cap A\cap I=\varnothing$ $~\forall ~u\in F$ where 
$$I=[z+C(l), z+\Lambda(D(l)h,l)-C(l)]=[z+C(l),z+C(l)+m(D(l)h,l)].$$
Hence,  Corollary
\ref{interval_gcd_in_Tl} gives us a sequence of  intervals in  $I$ (say $(S_i)_{i=1}^{D(l)h}$)
such that 
\begin{enumerate}[label=(\alph*)]
 \item 
  $|S_i|=2l+1$ and $ A\cap S_i \cap u\mathbb{N}=\varnothing$   $\forall ~i\in[1, D(l)h]$, $u\in F$,
\item 
$v_1\in S_i$, $v_2\in S_j$ and  $v_1\neq v_2$ for $i,j\in[1, D(l)h]$ $\Rightarrow $ $\gcd(v_1,v_2)\in T(l)$.
\end{enumerate}
Since $A$ is $(2l+1)$-syndetic set, there exists a sequence $(s_i)_{i=1}^{D(l)h}$ such that $s_i\in S_i\cap A$ and $\gcd(s_i,s_j)\in T(l)\setminus Mul(F)$ for $i\neq j$.
Define   $$g_i:=\max\{u\in T(l)\setminus Mul(F)\colon u\mid s_i\} ~\forall ~i\in [1, D(l)h]. $$ Using  $|T(l)|=D(l)$, there exist 
$J\subset [1, D(l)h]$ and $w\in T(l)\setminus Mul(F)$ 
such that $|J|\geq h$ and $g_j=w ~\forall ~j\in J.$
Set $C_w:=\left\{\frac{s_j}{w}\colon j\in J\right\}.$

Since $g_j=w ~\forall ~j\in J$ and $\gcd(s_{j_1},s_{j_2})\in T(l)\setminus Mul(F) ~\forall ~j_1,j_2\in J$ with $j_1\neq j_2$, we have
 $\gcd(s_{j_1},s_{j_2})=w ~\forall ~j_1,j_2\in J$ with $j_1\neq j_2$. Therefore, $C_w$ is a pairwise prime set and
$wC_w=\left\{s_j\colon j\in J\right\}\subset A.$
\end{proof}
\begin{prop}\label{Existence_triveni_higher}
 Suppose that $k,l\in\mathbb{N}$, $H_0$ be an infinite pairwise prime subset of $\mathbb{N}$ and  $A$ be a $(2l+1)$-syndetic set 
 with $A\notin \mathcal{S}_{k,2l+1,H_0}$.
If $(F,h,l)$ is an $\tvname$ for each $h\in \mathbb{N}$, then $\exists ~$ $F'\subset T(l)\setminus \{1\}$ with $F\subsetneq F'$ such that $(F',h,l)$ is an $\tvname$ for each $h\in \N$.
\end{prop}
\begin{proof}
 Let $k_0\in \mathbb{N}\setminus [1,2l+1]$. Since $(F,k_0,l)$ is an $\tvname$, there exists a sequence of pairwise prime sets
 $(B_u)_{u\in F}$ such that  for every distinct $u,u_1\in F$, $|B_u|=k_0$ and $\gcd(x,y)=1$ $\forall$ $x\in B_u, y\in B_{u_1}$. Let $B:= \cup_{u\in F}B_u$, $\alpha:=\{p\in\mathbb{P}\colon 
 p\mid x \text{ for some }x\in B\}$ and $W$ be the cardinality of $\alpha$.

 Since $A$ is a $(2l+1)$-syndetic set with $A\notin \mathcal{S}_{k,2l+1,H_0}$ and  $(F,h,l)$ is an  $\tvname$ for each $h\in \mathbb{N}$, 
 Proposition \ref{Higher_triveni_Main_prop}
 guarantees the existence of an element $v\in T(l)\setminus Mul(F)$ and a  pairwise prime set $C_v$ such that $|C_v|=W+k_0$ and
 $vC_v\subset A$. Also, Proposition \ref{main_cor1} ensures $v\neq 1$.
 Using the fact that $\alpha\subset \mathbb{P}$ with $|\alpha|=W$
and $C_v$ is a pairwise prime set of 
 cardinality $W+k_0$, we get a pairwise prime set $B_v\subset C_v$ such that $|B_v|=k_0$ and 
 $\gcd(p,b)=1$ $\forall$ $p\in \alpha, b\in B_v$. Then $\gcd(a,b)=1$ $\forall$ $a\in B, b\in B_v$ because elements of the pairwise prime set $B$ are made from primes in $\alpha$. Moreover $vB_v\subset A$, because $B_v\subset C_v$ and $vC_v\subset A$. 
 Since $(F,k_0,l)$ is an $\tvname$, therefore, $(F_{k_0},k_0,l)$ is also an $\tvname$ with $F_{k_0}=F\cup \{v\}\subset T(l)\setminus \{1\}$.
  
  Here we constructed a sequence $(F_n)_{n\in\N}$ such that $F\subsetneq  F_n\subset T(l)\setminus\{1\}$ and
  $(F_n,n,l)$ is an $\tvname ~\forall ~$ $n\in \mathbb{N}$. Since $|T(l)|< \infty$, there exist subsequence 
  $(F_{n_t})_{t\in \N}$ of  $(F_n)_{n\in\N}$ and set $F'$ with $F\subsetneq F'\subset T(l)\setminus \{1\}$ such 
  that $F_{n_t}=F' ~\forall ~$  $t\in\mathbb{N} $. Hence, $(F',t, l)$ is an  $\tvname$ for each 
  $t\in \mathbb{N}$.
 \end{proof}

 The  combination of Corollary \ref{Existence_triveni_1} and Proposition \ref{Existence_triveni_higher} generates Triveni triplets of various orders
with respect to those syndetic sets which do not contain configurations of the form $\{x, xn^kr\}$ where
$r\in \mathbb{N},n\in H_0$ with $r\equiv 1\pmod{n}$. Using these observations,  we will now see the  proof of Theorem \ref{thm1}.
\subsection{Proof of Theorem \ref{thm1}}\label{section4}
Let $H_0=\mathbb{P}$ or $H_0$ be an infinite pairwise prime subset of the set of composite numbers.
Since $A$ is a syndetic set, there exists $l\in \mathbb{N}$ such that
$A$ is a $(2l+1)$-syndetic set. 

For
$k\in \mathbb{N}$, if possible assume that  $A\notin \mathcal{S}_{k,2l+1,H_0}$. Then  Corollary \ref{Existence_triveni_1} gives the existence of an integer $d\in [2,l]$  such that $(\{d\},h,l)$ is an $\tvname$ for each $h\in \mathbb{N}$. Therefore, by Proposition
 \ref{Existence_triveni_higher}, there exists a sequence $(F_i)_{i=0}^{\infty}$ such that 
 \begin{enumerate}[label=(\alph*)]
  \item \label{ethitem2} for each $i\in [0,\infty)$ and $h\in \mathbb{N}$, $(F_i,h,l)$ is an $\tvname$,
 \item \label{ethitem1}$\{d\}=F_0\subsetneq F_1\subsetneq F_2\subsetneq \cdots \subsetneq F_i \subsetneq\cdots \subset T(l)$.
 \end{enumerate}
Since $|T(l)|<\infty$,  property \ref{ethitem1} of sequence  $(F_i)_{i=0}^{\infty}$ is a contradiction.
 Therefore, $A\in \mathcal{S}_{k,2l+1,H_0}$. In other words, the syndetic set $A$ contains configurations of the form $\{x, xn^kr\}$ where $r\in \mathbb{N},n\in H_0$
satisfying $r\equiv 1\pmod{n}$.

\section{Proof of Theorem \ref{thm2}}\label{section5}
Now we are going to prove Theorem \ref{thm2} by generating the required configurations using two different algorithms. 
\begin{lem}\label{thm2lemma1}
 A $2$-syndetic set $S$ contains infinitely many configurations of the type $\{x,xr^2\}$ or  infinitely many odd  perfect squares. 
\end{lem}
\begin{proof}
  Let $m\in \mathbb{N}$ be an odd integer. If $m^2+1\notin S$, then the $2$-syndeticity of $S$ ensures that the  odd perfect square $m^2\in S$. 
 On the other hand, if $m^2+1\in S$,
  then  the identity
 $(2m^2+1)^2-1=4m^2(m^2+1)$ guarantees that  $\{x,xr^2\}\subset S$ for $x=m^2+1$  and $r=2m$ whenever $(2m^2+1)^2-1\in S$. For the case $(2m^2+1)^2-1\notin S$, the odd perfect square
 $(2m^2+1)^2\in S$ 
 due to the $2$-syndeticity of $S$. Hence, infinitude of the odd integers completes the proof.

\end{proof}
\begin{lem}\label{thm2lemma2}
 If $y\in 2\mathbb{N}+1$ and $S$ is  a $2$-syndetic set such that $y^2(y^2+2i)\in S$ for each $i\in\{1,-1\}$, then  $S$ contains a configuration of the type $\{x,xr^2\}$
 with $x,r\in\mathbb{N}\setminus\{1\}$ and $x\geq \frac{y^2-1}{4}$.
\end{lem}
\begin{proof}
 If $y^2+2i\in S$ for some $i\in \{1,-1\}$, then $\{x,xr^2\}\subset S$ for $x=y^2+2i$ and $r=y$. 
 On the other hand, if
$ y^2+2i\notin S $ for each $i\in \{1,-1\}$, then $\{y^2-1,y^2+3\}\subset S$ due to $2$-syndeticity  of the set $S$. Since $y$ is an odd integer, we get consecutive natural numbers $a_y$ and $b_y$ satisfying
$4a_y=y^2-1 \in S \ \text{ and } \ 4b_y=y^2+3\in S.$ Since $S$
 is a $2$-syndetic set, so it follows that one of $a_y$ and $b_y$ lies inside $S$. Hence,  $S$ contains 
 $\{a_y, 4a_y\}$ or $\{b_y, 4b_y\}$. 
Therefore, $S$ contains the configuration $\{x,xr^2\}$ for $x\in\{a_y,b_y\}$ and $r=2$.
\end{proof}

\begin{proof}[Proof of Theorem \ref{thm2} (first method)]
Let $S$ be a $2$-syndetic set. By Lemma \ref{thm2lemma1}, it is enough to show that if  $m^2\in S$ for some $m\in 2\mathbb{N}+1$, then  
$S$ contains a configuration of the type $\{x,xr^2\}$ where $x,r\in \mathbb{N}\setminus\{1\}$ and $x\geq\frac{m-1}{4}$

Let $m\in \mathbb{N}\setminus \{1\}$ be an odd integer  such that $m^2\in S$. If $m^2n^2\in S$ for some $n\in \mathbb{N}\setminus\{1\}$, 
then we have $\{x,xr^2\}\subset S$ for $x=m^2$ and $r=n$. 
If  $m^2n^2\notin S$ for each $n\in \mathbb{N}\setminus\{1\}$, then the  $2$-syndeticity of $S$ ensures that
\begin{equation}
m^2n^2-1\in S ~~\forall ~~n\in \mathbb{N}\setminus\{1\}. \label{pfthm2eq1}
\end{equation}
Choose $z\in 2\mathbb{N}+1$ satisfying $z^2\equiv 0\pmod{m}$. 
Define $u(z,i):=z^2+i$ for each $i\in \{-1,1\}$.    
Clearly $u(z,i)^2\equiv 1\pmod{m}$ for each $i\in\{-1,1\}$.
Then, for each $i\in \{-1,1\}$,there exists $k_{u(z,i)}\in \mathbb{N}\setminus \{1\}$
\begin{equation}
mk_{u(z,i)}+1=u(z,i)^2. \label{pfthm2eq2}
\end{equation}
Putting $n=k_{u(z,i)}$ in expression \eqref{pfthm2eq1}, we get $m^2k_{u(z,i)}^2-1\in S$. Inserting the value of $k_{u(z,i)}$ from equation \eqref{pfthm2eq2} in this, we have
$m^2k_{u(z,i)}^2-1=({u(z,i)}^2-2){u(z,i)}^2\in S.$

Therefore, $\{x,xr^2\}\subset S$ for $x={u(z,i)}^2-2$ and $r={u(z,i)}$ whenever ${u(z,i)}^2-2\in S$.
For the remaining case  $u(z,i)^2-2\notin S \ \forall \ i\in\{-1,1\}$,  the $2$-syndeticity of $S$
 guarantees that
  $$u(z,i)^2-1=z^2(z^2+2i)\in S \ \forall \ i\in\{1,-1\}.$$
  Therefore we complete the proof by taking $y=z$ in Lemma \ref{thm2lemma2}.
\end{proof}
The above algorithm also guarantees a configuration of the form $\{x,xr^2\}$ with $r<x$ inside syndetic sets containing
the fourth power of some odd integer. For this, take $m$ to be an odd perfect square and choose $z=\sqrt{m}$ for some $i\in \{1,-1\}$ in the algorithm in the first method.
The next algorithm generates a configuration of the form $\{x,xr^2\}$ inside syndetic sets with the condition $r>x$ using the identity in the following lemma.
\begin{lem}\label{thm2lemma3}
 For $a\in \mathbb{N}$,  $a(4a+3)^2+1=(a+1) (4a+1)^2$.
\end{lem}
\begin{proof}[Proof of Theorem \ref{thm2} (second method)]
 Let $S$ be a $2$-syndetic set. If $\{a,a+1\}\subset S$ for infinitely many $a\in \mathbb{N}$, then  applying the identity in Lemma \ref{thm2lemma3},
we get that $\{x,xr^2\}\subset S$ for $(x,r)=(a, (4a+3))$ or $(a+1, (4a+1))$ for those $a'$s. On the other hand, $S$ contains an infinite arithmetic progression  whenever $\{a,a+1\}\subset S$ for only finitely many $a\in \mathbb{N}$.
Therefore, we finish the proof by using the fact that any infinite arithmetic progression  contains an infinite geometric progression.
 \end{proof}

% ------------------------------------------------------------------------
\end{document}